\documentclass{amsart}
\usepackage{amsfonts}
\usepackage{amsmath}
\usepackage{hyperref}
\usepackage{graphicx}

\newtheorem{theorem}{Theorem}
\theoremstyle{plain}

\newtheorem{remark}{Remark}

\numberwithin{equation}{section}

\input{tcilatex}

\begin{document}

\title[Two-dimensional Pompeiu's Mean Value Theorems] {Two-dimensional Pompeiu's Mean Value Theorems and Related Results}
\author[M.W. Alomari]{Mohammad W. Alomari}
\address{Department of Mathematics,
Faculty of Science and Information Technology, Irbid National
University, 2600 Irbid 21110, Jordan.} \email{mwomath@gmail.com}

\date{\today}


\begin{abstract}
In this work a mean value theorem of Pompeiu's type for functions
of two variables is presented. Other related results are given as
well.
\end{abstract}

\maketitle

\section{Introduction}
\emph{The mean-value theorem} (MVT) or Lagrange MVT is considered
as one of the most useful and fundamental result in analysis,
named after the French mathematician Joseph Lagrange, where he
presented his mean value theorem in his book Theorie des functions
analytiques in 1797; he states that: ``If $f$ is continuous on
$[a,b]$, differentiable on $(a,b)$ then there is a point $c$,
$a<c<b$, such that
\begin{align}
f'(c)=\frac{f(b)-f(a)}{b-a}.\text{"}
\end{align}
We call such a point $c$ a mean-value point of $f$. Further note
that the mean-value theorem of differentiation can be used to
prove that: if $f' >0$ and every sub-interval contains a point at
which $f' >0$, in particular if $f' >0$ with $f'(x )=0$ at only a
finite number of points, then $f$ is strictly increasing.

In recent years several authors modified and generalized various
types of mean value theorems in different ways and interesting
approach. For recent works the reader may refer for example to
\cite{Abela}, \cite{Tan}, \cite{Trokhimchuk}, for general reading
the interested researcher may find a hundreds of works cited in
the book \cite{Sahoo}.

In 1946, Pompeiu established another MVT for real functions
defined on a real interval that not containing `$0$'; nowadays
known as Pompeiu's mean value theorem, which states
\cite{Pompeiu}( see also, p., 83; \cite{Sahoo}):
\begin{theorem}
\label{pomp}For every real valued function $f$ differentiable on
an interval $[a,b]$ not containing $0$ and for all pairs $x_1\ne
x_2$ in $[a,b]$ there exists a point $\xi\in (x_1,x_2)$ such that
\begin{align}
\frac{{x_1 f\left( {x_2 } \right) - x_2 f\left( {x_1 }
\right)}}{{x_1  - x_2 }} = f\left( \xi \right) - \xi f'\left( \xi
\right).
\end{align}
\end{theorem}
The geometrical interpretation of this theorem as given in
\cite{Sahoo}: the tangent at the point $\left( {\xi ,f\left( \xi
\right)} \right)$ intersects on the $y$-axis at the same point as
the secant line connecting the points $ \left( {x_1 ,f\left( {x_1
} \right)} \right)$ and $ \left( {x_2 ,f\left( {x_2 } \right)}
\right)$. \\

In 1947, Boggio \cite{Boggio}  (see also, p.,92; \cite{Sahoo})
established the following generalization of Pompeiu's mean value
theorem \ref{pomp}:
\begin{theorem}
\label{Boggio}For every real valued functions $f$ and $g$
differentiable on an interval $[a,b]$ not containing $0$ and for
all pairs $x_1\ne x_2$ in $[a,b]$ there exists a point $\xi\in
(x_1,x_2)$ such that
\begin{align}
\label{eq.Boggio}\frac{{g\left( {x_1 } \right)f\left( {x_2 }
\right) - g\left( {x_2 } \right)f\left( {x_1 } \right)}}{{g\left(
{x_2 } \right) - g\left( {x_1 } \right)}} = f\left( \xi  \right) -
\frac{{g\left( \xi  \right)}}{{g'\left( \xi  \right)}}f'\left( \xi
\right).
\end{align}
\end{theorem}

In their famous book \cite{Sahoo}, Sahoo and Riedel have discussed
various type of mean value theorems for functions of one or more
variables. Among others,  they  stated in the end of Chapter four
(p., 145; \cite{Sahoo}) that: ``\emph{We have not been able to
generalize Pompeiu's mean value theorem for functions in two
variables}...".

The aim of this work is to answer Sahoo and Riedel problem about
the characterization of Pompeiu's MVT for functions of two
variables. Namely, for functions of two variables; we prove three
mean value theorems; the Cauchy--, Pompeiu's--, and
Cauchy--Pompeiu's--mean value theorems.\\

Two useful results concerning Rectangular MVT for functions of two
variables have been recently obtained  in (p., 93;
\cite{Ghorpade}):
\begin{theorem}(Rectangular Rolle's Theorem)
\label{RRT}Let $a,b,c,d \in \mathbb{R}$ with $a<b$ and $c<d$, and
let $f:\Delta:= [a,b] \times[c,d]$ satisfy the following:
\begin{enumerate}
\item For each fixed $y_0 \in [c,d]$, the function given by $x
\mapsto f\left( {x,y_0 } \right)$ is continuous on $[a,b]$ and
differentiable on $(a,b)$.

\item For each fixed $x_0 \in [a,b]$, the function given by $y
\mapsto f_x\left( {x_0,y} \right)$ is continuous on $[c,d]$ and
differentiable on $(c,d)$.

\item $f(a,c)+ f(b,d)=f(a,d)+f(b,c)$.
\end{enumerate}
Then, there exists $(x_0,y_0) \in (a,b)\times (c,d)$ such that
$f_{xy}(x_0,y_0)=0$.
\end{theorem}

\begin{theorem}(Rectangular Mean Value Theorem)
\label{RMVT}Let $a,b,c,d \in \mathbb{R}$ with $a<b$ and $c<d$, and
let $f:\Delta:= [a,b] \times[c,d]$ satisfy the following:
\begin{enumerate}
\item For each fixed $y_0 \in [c,d]$, the function given by $x
\mapsto f\left( {x,y_0 } \right)$ is continuous on $[a,b]$ and
differentiable on $(a,b)$.

\item For each fixed $x_0 \in [a,b]$, the function given by $y
\mapsto f_x\left( {x_0,y} \right)$ is continuous on $[c,d]$ and
differentiable on $(c,d)$.
\end{enumerate}
Then, there exists $(x_0,y_0) \in (a,b)\times (c,d)$ such that
\begin{eqnarray*}
f\left( {b,d} \right) - f\left( {b,c} \right) - f\left( {a,d}
\right) + f\left( {a,c} \right) = (b-a)(d-c)f_{xy}(x_0,y_0).
\end{eqnarray*}
\end{theorem}

\section{The Result}

We begin with the following generalization of Theorem \ref{RRT}
and Theorem \ref{RMVT}:
\begin{theorem}(Rectangular Cauchy Mean Value Theorem)
\label{RCMVT}Let $a,b,c,d \in \mathbb{R}$ with $a<b$ and $c<d$,
and let $f,g:\Delta:= [a,b] \times[c,d]$ satisfy the following:
\begin{enumerate}
\item For each fixed $y_0 \in [c,d]$, the functions given by $x
\mapsto f\left( {x,y_0 } \right)$ and $x \mapsto g\left( {x,y_0 }
\right)$ are continuous on $[a,b]$ and differentiable on $(a,b)$.

\item For each fixed $x_0 \in [a,b]$, the functions given by $y
\mapsto f_x\left( {x_0,y} \right)$ and $x \mapsto g_x\left( {x_,y
} \right)$ are continuous on $[c,d]$ and differentiable on
$(c,d)$.
\end{enumerate}
Then, there exists $(x_0,y_0) \in (a,b)\times (c,d)$ such that
\begin{eqnarray}
\label{eq.RCMVT}\frac{{f_{xy} \left( {x_0 ,y_0 } \right)}}{{g_{xy}
\left( {x_0 ,y_0 } \right)}} = \frac{{f\left( {b,d} \right) -
f\left( {b,c} \right) - f\left( {a,d} \right) + f\left( {a,c}
\right)}}{{g\left( {b,d} \right) - g\left( {b,c} \right) - g\left(
{a,d} \right) + g\left( {a,c} \right)}}.
\end{eqnarray}
\end{theorem}

\begin{proof}
Define the function $H:\Delta \to \mathbb{R}$, given by
\begin{multline*}
 H\left( {x,y} \right) = \left[ {f\left( {b,d} \right) - f\left( {b,c} \right) - f\left( {a,d} \right) + f\left( {a,c} \right)} \right]g\left({x,y}\right)
\\
- \left[ {g\left( {b,d} \right) - g\left( {b,c} \right) - g\left(
{a,d} \right) + g\left( {a,c} \right)} \right]f\left( {x,y}
\right).
\end{multline*}
It is easy to see that $H$ is continuous and differentiable on
$D$, and
\begin{multline*}
H\left( {a,c} \right) = \left[ {f\left( {b,d} \right) - f\left(
{b,c} \right) - f\left( {a,d} \right) + f\left( {a,c} \right)}
\right]g\left( {a,c} \right)
\\
- \left[ {g\left( {b,d} \right) - g\left( {b,c} \right) - g\left(
{a,d} \right) + g\left( {a,c} \right)} \right]f\left( {a,c}
\right),
\end{multline*}
\begin{multline*}
H\left( {a,d} \right) = \left[ {f\left( {b,d} \right) - f\left(
{b,c} \right) - f\left( {a,d} \right) + f\left( {a,c} \right)}
\right]g\left( {a,d} \right)
\\
- \left[ {g\left( {b,d} \right) - g\left( {b,c} \right) - g\left(
{a,d} \right) + g\left( {a,c} \right)} \right]f\left( {a,d}
\right),
\end{multline*}
\begin{multline*}
H\left( {b,c} \right) = \left[ {f\left( {b,d} \right) - f\left(
{b,c} \right) - f\left( {a,d} \right) + f\left( {a,c} \right)}
\right]g\left( {b,c} \right)
\\
- \left[ {g\left( {b,d} \right) - g\left( {b,c} \right) - g\left(
{a,d} \right) + g\left( {a,c} \right)} \right]f\left( {b,c}
\right),
\end{multline*}
and
\begin{multline*}
H\left( {b,d} \right) = \left[ {f\left( {b,d} \right) - f\left(
{b,c} \right) - f\left( {a,d} \right) + f\left( {a,c} \right)}
\right]g\left( {b,d} \right)
\\
- \left[ {g\left( {b,d} \right) - g\left( {b,c} \right) - g\left(
{a,d} \right) + g\left( {a,c} \right)} \right]f\left( {b,d}
\right),
\end{multline*}
then we have
\begin{align*}
&H\left( {a,c} \right) - H\left( {a,d} \right) - H\left( {b,c}
\right) + H\left( {b,d} \right)
\\
&= \left[ {f\left( {b,d} \right) - f\left( {b,c} \right) - f\left(
{a,d} \right) + f\left( {a,c} \right)} \right]\left[ {g\left(
{a,c} \right) - g\left( {a,d} \right) - g\left( {b,c} \right) +
g\left( {b,d} \right)} \right]
\\
&\qquad- \left[ {g\left( {a,c} \right) - g\left( {a,d} \right) -
g\left( {b,c} \right) + g\left( {b,d} \right)} \right]\left[
{f\left( {b,d} \right) - f\left( {b,c} \right) - f\left( {a,d}
\right) + f\left( {a,c} \right)} \right]
\\
&= 0 ,
\end{align*}
which gives that $H\left( {a,d} \right) + H\left( {b,c}
\right)=H\left( {a,c} \right) + H\left( {b,d} \right)$. So by the
Rectangular Rolle's Theorem \ref{RRT}, there is $(x_0, y_0) \in
(a,b) \times (c,d)$ such that $H_{xy}(x_0,y_0) = 0$, therefore
\begin{multline*}
H_{xy} \left( {x_0 ,y_0 } \right) = \left[ {f\left( {b,d} \right)
- f\left( {b,c} \right) - f\left( {a,d} \right) + f\left( {a,c}
\right)} \right]g_{xy} \left( {x_0 ,y_0 } \right)
\\
- \left[ {g\left( {b,d} \right) - g\left( {b,c} \right) - g\left(
{a,d} \right) + g\left( {a,c} \right)} \right]f_{xy} \left( {x_0
,y_0 } \right) = 0,
\end{multline*}
which gives that
\begin{eqnarray*}
\frac{{f_{xy} \left( {x_0 ,y_0 } \right)}}{{g_{xy} \left( {x_0
,y_0 } \right)}} = \frac{{f\left( {b,d} \right) - f\left( {b,c}
\right) - f\left( {a,d} \right) + f\left( {a,c} \right)}}{{g\left(
{b,d} \right) - g\left( {b,c} \right) - g\left( {a,d} \right) +
g\left( {a,c} \right)}}.
\end{eqnarray*}
This yields the desired result.
\end{proof}

\begin{remark}
In Theorem \ref{RCMVT}, if one chooses, $g(t,s)=ts$, then we
recapture the rectangular Mean-Value Theorem \ref{RMVT}.
\end{remark}

Our first main  result concerning Pompeiu's Mean Value Theorem for
functions of two variables may be considered as follows:

\begin{theorem}(Pompeiu's Mean Value Theorem)
\label{thm4}Let $a,b,c,d \in \mathbb{R}$ with $a<b$ and $c<d$, and
let $f:\Delta:= [a,b] \times[c,d]\to \mathbb{R}$ satisfy the
following:
\begin{enumerate}
\item $\Delta$ not containing the points $(0,\cdot), (\cdot,0),
(0,0)$.

\item For each fixed $y_0 \in [c,d]$, the function given by $x
\mapsto f\left( {x,y_0 } \right)$ is continuous on $[a,b]$ and
differentiable on $(a,b)$.

\item For each fixed $x_0 \in [a,b]$, the function given by $y
\mapsto f_x\left( {x_0,y} \right)$ is continuous on $[c,d]$ and
differentiable on $(c,d)$.

\item For all pair $x_1,x_2 \in(a,b)$ with $x_1 \ne x_2$ and
$y_1,y_2 \in (c,d)$ with $y_1 \ne y_2$.
\end{enumerate}
Then, there exists $(\xi_1,\xi_2) \in (x_1,y_1)\times (x_2,y_2)$
such that
\begin{multline}
\label{eq2.2}\xi _1 \xi _2 \frac{{\partial ^2 f}}{{\partial
t\partial s}}\left( {\xi _1 ,\xi _2 } \right) - \xi _1
\frac{{\partial f}}{{\partial t}}\left( {\xi _1 ,\xi _2 } \right)
- \xi _2 \frac{{\partial f}}{{\partial s}}\left( {\xi _1 ,\xi _2 }
\right) + f\left( {\xi _1 ,\xi _2 } \right)
\\
= \frac{{x_2 y_2 f\left( {x_1 ,y_1 } \right) - x_2 y_1 f\left(
{x_1 ,y_2 } \right) - x_1 y_2 f\left( {x_2 ,y_1 } \right) + x_1
y_1 f\left( {x_2 ,y_2 } \right)}}{{\left( {x_2  - x_1 }
\right)\left( {y_2  - y_1 } \right)}}.
\end{multline}
\end{theorem}

\begin{proof}
Define a real valued function $F:\left[ {\frac{1}{b},\frac{1}{a}}
\right] \times \left[ {\frac{1}{d},\frac{1}{c}} \right] \to
\mathbb{R}$, given by
\begin{align}
\label{1}F\left( {t,s} \right) = tsf\left(
{\frac{1}{t},\frac{1}{s}} \right).
\end{align}
By the assumptions (1)-(3), its easy to see that
\begin{enumerate}
\item $F$ is defined on $\left[ {\frac{1}{b},\frac{1}{a}} \right]
\times \left[ {\frac{1}{d},\frac{1}{c}} \right]$, since $\Delta$
does not containing the points $(0,\cdot), (\cdot,0), (0,0)$.

\item For each fixed $y_0 \in \left[ {\frac{1}{d},\frac{1}{c}}
\right]$, the function given by $x \mapsto F\left( {x,y_0 }
\right)$ is continuous on $\left[ {\frac{1}{b},\frac{1}{a}}
\right] $ and differentiable on $\left( {\frac{1}{b},\frac{1}{a}}
\right)$.

\item For each fixed $x_0 \in \left[ {\frac{1}{b},\frac{1}{a}}
\right]$, the function given by $y \mapsto F_x\left( {x_0,y}
\right)$ is continuous on $\left[ {\frac{1}{d},\frac{1}{c}}
\right]$ and differentiable on $\left( {\frac{1}{d},\frac{1}{c}}
\right)$.
\end{enumerate}
Therefore, simple calculations yield that
\begin{align*}
F_t \left( {t,s} \right) = sf\left( {\frac{1}{t},\frac{1}{s}}
\right) - \frac{s}{t}f_t \left( {\frac{1}{t},\frac{1}{s}} \right),
\end{align*}
\begin{align*}
F_s \left( {t,s} \right) = tf\left( {\frac{1}{t},\frac{1}{s}}
\right) - \frac{t}{s}f_s \left( {\frac{1}{t},\frac{1}{s}} \right),
\end{align*}
and
\begin{align}
\label{2}F_{ts} \left( {t,s} \right) &= \frac{1}{{ts}}f_{st}
\left( {\frac{1}{t},\frac{1}{s}} \right) - \frac{1}{t}f_t \left(
{\frac{1}{t},\frac{1}{s}} \right) - \frac{1}{s}f_s \left(
{\frac{1}{t},\frac{1}{s}} \right) + f\left(
{\frac{1}{t},\frac{1}{s}} \right)
\nonumber\\
&= F_{st} \left( {t,s} \right).
\end{align}
Applying the Rectangular Mean Value Theorem \ref{RMVT} to $F$ on
the interval $\left[ {u,v} \right] \times \left[ {z,w} \right]
\subset  \left[ {\frac{1}{b},\frac{1}{a}} \right] \times \left[
{\frac{1}{d},\frac{1}{c}} \right]$, we get
\begin{align}
\label{3}\left( {v - u} \right)\left( {w - z} \right)F_{ts} \left(
{\eta _1 ,\eta _2 } \right) = F\left( {v,w} \right) - F\left(
{v,z} \right) - F\left( {u,w} \right) + F\left( {u,z} \right)
\end{align}
for some $\left( {\eta _1 ,\eta _2 } \right) \in \left( {u,v}
\right) \times \left( {z,w} \right)$. Let $x_1  = \frac{1}{v}$,
$x_2  = \frac{1}{u}$, $y_1  = \frac{1}{w}$, $y_2  = \frac{1}{z}$,
$\xi _1  = \frac{1}{{\eta _1 }}$, and $\xi _2  = \frac{1}{{\eta _2
}}$. Then, since $\left( {\eta _1 ,\eta _2 } \right) \in \left(
{u,v} \right) \times \left( {z,w} \right)$, we have
$$x_1  \le \xi _1  \le x_2, \,\,\, \text{and}\,\,\,y_1  \le \xi _2  \le y_2.$$

Now, using (\ref{1}) and (\ref{2}) on (\ref{3}), we have
\begin{multline*}
\frac{1}{{\eta _1 \eta _2 }}f_{st} \left( {\frac{1}{{\eta _1
}},\frac{1}{{\eta _2 }}} \right) - \frac{1}{{\eta _1 }}f_t \left(
{\frac{1}{{\eta _1 }},\frac{1}{{\eta _2 }}} \right) -
\frac{1}{{\eta _2 }}f_s \left( {\frac{1}{{\eta _1
}},\frac{1}{{\eta _2 }}} \right) + f\left( {\frac{1}{{\eta _1
}},\frac{1}{{\eta _2 }}} \right)
\\
= \frac{1}{{\left( {v - u} \right)\left( {w - z} \right)}}\left[
{vwf\left( {\frac{1}{v},\frac{1}{w}} \right) - vzf\left(
{\frac{1}{v},\frac{1}{z}} \right) - uwf\left(
{\frac{1}{u},\frac{1}{w}} \right) + uzf\left(
{\frac{1}{u},\frac{1}{z}} \right)} \right],
\end{multline*}
that is,
\begin{multline*}
\xi _1 \xi _2 \frac{{\partial ^2 f}}{{\partial t\partial s}}\left(
{\xi _1 ,\xi _2 } \right) - \xi _1 \frac{{\partial f}}{{\partial
t}}\left( {\xi _1 ,\xi _2 } \right) - \xi _2 \frac{{\partial
f}}{{\partial s}}\left( {\xi _1 ,\xi _2 } \right) + f\left( {\xi
_1 ,\xi _2 } \right)
\\
= \frac{{x_2 y_2 f\left( {x_1 ,y_1 } \right) - x_2 y_1 f\left(
{x_1 ,y_2 } \right) - x_1 y_2 f\left( {x_2 ,y_1 } \right) + x_1
y_1 f\left( {x_2 ,y_2 } \right)}}{{\left( {x_2  - x_1 }
\right)\left( {y_2  - y_1 } \right)}}.
\end{multline*}
This completes the proof of the theorem.
\end{proof}

Next, a characterization of Boggio MVT which is of Cauchy's type
for functions of two variables is stated as follows:

\begin{theorem}(Boggio Mean Value Theorem)
Let $a,b,c,d \in \mathbb{R}$ with $a<b$ and $c<d$, and let
$f,g:\Delta:= [a,b] \times[c,d]\to \mathbb{R}$ satisfy the
following:
\begin{enumerate}
\item $\Delta$ not containing the points $(0,\cdot), (\cdot,0),
(0,0)$.

\item For each fixed $y_0 \in [c,d]$, the functions given by $x
\mapsto f\left( {x,y_0 } \right)$ and $x \mapsto g\left( {x,y_0 }
\right)$ are continuous on $[a,b]$ and differentiable on $(a,b)$.

\item For each fixed $x_0 \in [a,b]$, the functions given by $y
\mapsto f_x\left( {x_0,y} \right)$ and $y \mapsto g_x\left(
{x_0,y} \right)$ are continuous on $[c,d]$ and differentiable on
$(c,d)$.

\item For all pair $x_1,x_2 \in(a,b)$ with $x_1 \ne x_2$ and
$y_1,y_2 \in (c,d)$ with $y_1 \ne y_2$.
\end{enumerate}
Then, there exists $(\xi_1,\xi_2) \in (x_1,y_1)\times (x_2,y_2)$
such that
\begin{multline}
\label{eq2.6} \frac{{\xi _1 \xi _2 \frac{{\partial ^2
g}}{{\partial t\partial s}}\left( {\xi _1 ,\xi _2 } \right) - \xi
_1 \frac{{\partial g}}{{\partial t}}\left( {\xi _1 ,\xi _2 }
\right) - \xi _2 \frac{{\partial g}}{{\partial s}}\left( {\xi _1
,\xi _2 } \right) + g\left( {\xi _1 ,\xi _2 } \right)}}{{g\left(
{x_2 ,y_2 } \right) - g\left( {x_2 ,y_1 } \right) - g\left( {x_1
,y_2 } \right) + g\left( {x_1 ,y_1 } \right)}}
\\
- \frac{{\xi _1 \xi _2 \frac{{\partial ^2 f}}{{\partial t\partial
s}}f\left( {\xi _1 ,\xi _2 } \right) - \xi _1 \frac{{\partial
f}}{{\partial t}}f\left( {\xi _1 ,\xi _2 } \right) - \xi _2
\frac{{\partial f}}{{\partial s}}f\left( {\xi _1 ,\xi _2 } \right)
+ f\left( {\xi _1 ,\xi _2 } \right)}}{{f\left( {x_2 ,y_2 } \right)
- f\left( {x_2 ,y_1 } \right) - f\left( {x_1 ,y_2 } \right) +
f\left( {x_1 ,y_1 } \right)}}
\\
= \frac{{x_2 y_2 g\left( {x_1 ,y_1 } \right) - x_1 y_2 g\left(
{x_2 ,y_1 } \right) - x_2 y_1 g\left( {x_1 ,y_2 } \right) + x_1
y_1 g\left( {x_2 ,y_2 } \right)}}{{\left( {x_2  - x_1 }
\right)\left( {y_2  - y_1 } \right)\left[ {g\left( {x_2 ,y_2 }
\right) - g\left( {x_2 ,y_1 } \right) - g\left( {x_1 ,y_2 }
\right) + g\left( {x_1 ,y_1 } \right)} \right]}}
\\
- \frac{{x_2 y_2 f\left( {x_1 ,y_1 } \right) - x_2 y_1 f\left(
{x_1 ,y_2 } \right) - x_1 y_2 f\left( {x_2 ,y_1 } \right) + x_1
y_1 f\left( {x_2 ,y_2 } \right)}}{{\left( {x_2  - x_1 }
\right)\left( {y_2  - y_1 } \right)\left[ {f\left( {x_2 ,y_2 }
\right) - f\left( {x_2 ,y_1 } \right) - f\left( {x_1 ,y_2 }
\right) + f\left( {x_1 ,y_1 } \right)} \right]}}
\end{multline}

\end{theorem}

\begin{proof}
As in the proof of Theorem \ref{thm4}, let $\left[ {u,v} \right]
\times \left[ {z,w} \right] \subset  \left[
{\frac{1}{b},\frac{1}{a}} \right] \times \left[
{\frac{1}{d},\frac{1}{c}} \right]$, and setting $x_1  =
\frac{1}{v}$, $x_2  = \frac{1}{u}$, $y_1  = \frac{1}{w}$, $y_2  =
\frac{1}{z}$. Define the following two real valued functions
$H:\Delta \to \mathbb{R}$, given by
\begin{multline}
\label{4} H\left( {t,s} \right) = \left[ {f\left( {x_1,y_1}
\right) - f\left( {x_1,y_2} \right) - f\left( {x_2,y_1} \right) +
f\left( {x_2,y_2} \right)} \right]g\left({t,s}\right)
\\
- \left[ {g\left( {x_1,y_1} \right) - g\left( {x_1,y_2} \right) -
g\left( {x_2,y_1} \right) + g\left( {x_2,y_2} \right)}
\right]f\left( {t,s} \right),
\end{multline}
and $F:\left[ {\frac{1}{b},\frac{1}{a}} \right] \times \left[
{\frac{1}{d},\frac{1}{c}} \right] \to \mathbb{R}$, given by
\begin{align}
\label{5}F\left( {t,s} \right) = ts H\left(
{\frac{1}{t},\frac{1}{s}} \right)
\end{align}
By the assumptions (1)-(3), it is easy to see that
\begin{enumerate}
\item $F$ is defined on $\left[ {\frac{1}{b},\frac{1}{a}} \right]
\times \left[ {\frac{1}{d},\frac{1}{c}} \right]$, since $\Delta$
does not containing the points $(0,\cdot), (\cdot,0), (0,0)$.

\item For each fixed $y_0 \in \left[ {\frac{1}{d},\frac{1}{c}}
\right]$, the function given by $x \mapsto F\left( {x,y_0 }
\right)$ is continuous on $\left[ {\frac{1}{b},\frac{1}{a}}
\right] $ and differentiable on $\left( {\frac{1}{b},\frac{1}{a}}
\right)$.

\item For each fixed $x_0 \in \left[ {\frac{1}{b},\frac{1}{a}}
\right]$, the function given by $y \mapsto F_x\left( {x_0,y}
\right)$ is continuous on $\left[ {\frac{1}{d},\frac{1}{c}}
\right]$ and differentiable on $\left( {\frac{1}{d},\frac{1}{c}}
\right)$.
\end{enumerate}
Therefore, simple calculations yield that
\begin{align}
\label{6}F_{ts} \left( {t,s} \right) &= \frac{1}{{ts}}H_{st}
\left( {\frac{1}{t},\frac{1}{s}} \right) - \frac{1}{t}H_t \left(
{\frac{1}{t},\frac{1}{s}} \right) - \frac{1}{s}H_s \left(
{\frac{1}{t},\frac{1}{s}} \right) + H\left(
{\frac{1}{t},\frac{1}{s}} \right)
\nonumber\\
&= F_{st} \left( {t,s} \right).
\end{align}
Applying the Rectangular Mean Value Theorem \ref{RMVT} to $F$ on
the interval $\left[ {u,v} \right] \times \left[ {z,w} \right]
\subset  \left[ {\frac{1}{b},\frac{1}{a}} \right] \times \left[
{\frac{1}{d},\frac{1}{c}} \right]$, we get
\begin{align}
\label{7}\left( {v - u} \right)\left( {w - z} \right)F_{ts} \left(
{\eta _1 ,\eta _2 } \right) = F\left( {v,w} \right) - F\left(
{v,z} \right) - F\left( {u,w} \right) + F\left( {u,z} \right).
\end{align}
for some $\left( {\eta _1 ,\eta _2 } \right) \in \left( {u,v}
\right) \times \left( {z,w} \right)$.

Using the assumption that $x_1  = \frac{1}{v}$, $x_2  =
\frac{1}{u}$, $y_1  = \frac{1}{w}$, $y_2  = \frac{1}{z}$, $\xi _1
= \frac{1}{{\eta _1 }}$, and $\xi _2  = \frac{1}{{\eta _2 }}$.
Then, since $\left( {\eta _1 ,\eta _2 } \right) \in \left( {u,v}
\right) \times \left( {z,w} \right)$, we have
$$x_1  \le \xi _1  \le x_2, \,\,\, \text{and}\,\,\,y_1  \le \xi _2  \le y_2.$$

Now, using (\ref{4})--(\ref{6}) on (\ref{7}), we have
\begin{multline*}
\frac{{\xi _1 \xi _2 \frac{{\partial ^2 g}}{{\partial t\partial
s}}\left( {\xi _1 ,\xi _2 } \right) - \xi _1 \frac{{\partial
g}}{{\partial t}}\left( {\xi _1 ,\xi _2 } \right) - \xi _2
\frac{{\partial g}}{{\partial s}}\left( {\xi _1 ,\xi _2 } \right)
+ g\left( {\xi _1 ,\xi _2 } \right)}}{{g\left( {x_2 ,y_2 } \right)
- g\left( {x_2 ,y_1 } \right) - g\left( {x_1 ,y_2 } \right) +
g\left( {x_1 ,y_1 } \right)}}
\\
- \frac{{\xi _1 \xi _2 \frac{{\partial ^2 f}}{{\partial t\partial
s}}f\left( {\xi _1 ,\xi _2 } \right) - \xi _1 \frac{{\partial
f}}{{\partial t}}f\left( {\xi _1 ,\xi _2 } \right) - \xi _2
\frac{{\partial f}}{{\partial s}}f\left( {\xi _1 ,\xi _2 } \right)
+ f\left( {\xi _1 ,\xi _2 } \right)}}{{f\left( {x_2 ,y_2 } \right)
- f\left( {x_2 ,y_1 } \right) - f\left( {x_1 ,y_2 } \right) +
f\left( {x_1 ,y_1 } \right)}}
\\
= \frac{{x_2 y_2 g\left( {x_1 ,y_1 } \right) - x_1 y_2 g\left(
{x_2 ,y_1 } \right) - x_2 y_1 g\left( {x_1 ,y_2 } \right) + x_1
y_1 g\left( {x_2 ,y_2 } \right)}}{{\left( {x_2  - x_1 }
\right)\left( {y_2  - y_1 } \right)\left[ {g\left( {x_2 ,y_2 }
\right) - g\left( {x_2 ,y_1 } \right) - g\left( {x_1 ,y_2 }
\right) + g\left( {x_1 ,y_1 } \right)} \right]}}
\\
- \frac{{x_2 y_2 f\left( {x_1 ,y_1 } \right) - x_2 y_1 f\left(
{x_1 ,y_2 } \right) - x_1 y_2 f\left( {x_2 ,y_1 } \right) + x_1
y_1 f\left( {x_2 ,y_2 } \right)}}{{\left( {x_2  - x_1 }
\right)\left( {y_2  - y_1 } \right)\left[ {f\left( {x_2 ,y_2 }
\right) - f\left( {x_2 ,y_1 } \right) - f\left( {x_1 ,y_2 }
\right) + f\left( {x_1 ,y_1 } \right)} \right]}}.
\end{multline*}
This completes the proof of the theorem.
\end{proof}

\begin{remark}
In (\ref{eq2.6}), if one chooses $g(t,s)=ts$, then we recapture
the Pompeiu's Mean Value Theorem \ref{thm4}.
\end{remark}

\centerline{}

\centerline{}


\begin{thebibliography}{9}
\setlength{\itemsep}{5pt}

\bibitem{Abela}
Ulrich Abela, Mircea Ivanb and Thomas Riedelc, The mean value
theorem of Flett and divided differences,  \textit{Journal of
Mathematical Analysis and Applications}, 295 (1) (2004), 1--9.

\bibitem{Boggio}
T. Boggio, Sur une proposition de M. Pompeiu.,
\textit{Mathematica} (Cluj), 23 (1947), 101--102.

\bibitem{Ghorpade} S. Ghorpade and B. Limaye, A Course in Multivariable Calculus and
Analysis, Springer, New York, 2009.

\bibitem{Pompeiu}
D. Pompeiu, Sur une proposition analogue au theoreme des
accroissements finis, \textit{Mathematica} (Cluj), 22 (1946)
143--146.



\bibitem{Sahoo} P. Sahoo and T. Riedel, Mean Value Theorems and Functional Equations,
Singapore, World Scientific Publishing Co. Pte. Ltd. 1998.


\bibitem{Tan}
Chengguan Tan and Songxiao Li, Some new mean value theorems of
Flett type, \textit{International Journal of Mathematical
Education in Science and Technology}, 45 (7) (2014), 1103--1107.

\bibitem{Trokhimchuk}
Yu. Yu. Trokhimchuk, Mean-value theorem, \textit{Ukrainian
Mathematical Journal}, Translated version (Ukrainian Original Vol.
65, No. 9, September, 2013), 65 (9) (2014),  1418--1425.

\end{thebibliography}
\end{document}